\documentclass[oneside]{article}

\usepackage{latexsym,amsmath,amssymb,amsthm}

\newtheorem{theorem}{Theorem}[section]
\newtheorem{corollary}{Corollary}[section]
\newtheorem{lemma}{Lemma}[section]
\newtheorem{proposition}{Proposition}[section]

\begin{document}
\numberwithin{equation}{section}
\title{The $p$-adic Analysis of Stirling Numbers 
via Higher Order Bernoulli Numbers}
\author{Arnold Adelberg\\
Department of Mathematics and Statistics\\
Grinnell College\\
Grinnell, IA 50112}
\maketitle

\vspace{-10mm}

\begin{abstract}
In this paper, we use our previous study of the higher order Bernoulli numbers
$B_n^{(l)}$ to investigate $p$-adic properties of Stirling numbers of the second kind
$S(n,k)$. For example we give a new greatly simplified proof of the formula
$\nu_2(S(2^h,k))=d_2(k)-1$ if $1 \le k \le 2^h$, and generalize this result to arbitrary 
primes $p$. We also consider the Stirling numbers of the first kind
$s(n,k)$, with new results analogous to those for the Stirling numbers of the
second kind. New mod $p$ congruences for Stirling numbers of both kinds are also
given. 

\end{abstract}

\section{Introduction}

The starting point of our investigation was the remarkable formula conjectured by 
T. Lengyel [12] that
\begin{align}
\nu_2(S(2^h,k))=\sigma_2(k)-1
\end{align}
if $1 \le k \le 2^h$, where $S(n,k)=$ Stirling number of the second kind, $\nu_2=$2-adic
valuation, and $\sigma_2=$ base $2$ digit sum = number of base $2$ digits.

This formula was conjectured by Lengyel in 1994 and proven by S. De Wannemacker [7] in
2005. De Wannemacker's proof is quite involved. Furthermore the proof appears to be
only suitable for the prime $p=2$. 

We were surprised to observe that De Wannemacker's Theorem is an immediate consequence 
of our previous study of higher order Bernoulli numbers and polynomials, primarily of
the pole structure, which we developed in a series of papers in the nineties [1,2,3,4]. 
The machinery
of these papers is valid for all primes $p$, and enables us to extend De Wannemacker's 
Theorem to odd primes $p$ without additional effort. We also get a significant 
improvement of this theorem, which is new even for $p=2$.

Although the connection between higher order Bernoulli numbers and Stirling numbers
\begin{align}
S(n,k)=\binom{n}{k}B_{n-k}^{(-k)} \quad \text{and} \quad s(n,k)=\binom{n-1}{k-1}
 B_{n-k}^{(n)}
\end{align}
is well known and has been noted in [2, 3], we have not previously pursued this application 
in any depth.

Subsequent to De Wannemacker's proof, Lengyel used the same methods to strengthen his
original conjecture to prove [13] that  
\begin{align}
\nu_2(S(c2^h,k))=\sigma_2(k)-1
\end{align}
if $c\ge 1$ and $1 \le k \le 2^h$. 

We will not prove this stronger result in this paper, but we will prove it in a 
subsequent paper, along with its generalization to arbitrary primes. If  $\sigma_2(c)>1$,
these Stirling numbers do not have the ``minimum zero property," which is the main focus 
of this paper.

In addition, we have found greatly simplified proofs for other important results on 
Stirling numbers of the second kind, e.g. we have a nice proof of the theorem proven
by O-Y. Chan and D. Manna [6, Th. 2.4] that the central Stirling number $S(2k,k)$ is odd if
and only if $k$ is Fibbinary (i.e., the base $2$ representation of $k$ has no consecutive
ones). We also present a generalization valid for all primes $p$, namely we determine when
$p\nmid S(pk,k)$, using a simple analog of the Fibbinary property. We also give a new
mod $p$ congruence for $S(pk,k)$, which contains additional information if $p \ne 2$. 

We have abstracted the role of $2^h$ in De Wannemacker's Theorem to the ``minimum zero
property," and have used this concept to strengthen the result of T. Amdeberhan et al [5], 
conjectured in 2008 and proven by S. Hong et al [11, Th. 3.2] in 2012, that
\begin{align}
\nu_2(S(2^h+1,k+1))=\sigma_2(k)-1 \quad \text{if~} 1 \le k \le 2^h.
\end{align}

This is also generalized to all primes, as well as to all ``minimum zero cases."

In all instances where we have been able to exactly determine $\nu_p(S(n,k))$, we have also
been able to find simple explicit mod $p$ congruences for $\epsilon_p(S(n,k))=
p^{-\nu_p(S(n,k))}S(n,k)$, which is the part of $S(n,k)$ prime to $p$.

In our subsequent paper, we will also consider some cases which are not ``minimum zero
cases."  We have tried to incorporate
enough material in our background section to facilitate this extension.

We also consider the Stirling numbers of the first kind $s(n,k)$. The 
``minimum zero property" now necessitates that $k \le n < kp$ in addition to
$p-1\, |\, n-k$.
We use this property to prove an analog of DeWannemacker's Theorem, that
$\nu_2(s(n,2^h)) = h-\sigma_2(n-1)$ if $2^h \le n < 2^{h+1}$, and we generalize this
result to arbitrary primes. 

Similarly we have an analog of the Hong, Zhao and Zhao result for Stirling numbers of the 
first kind, that $\nu_2(s(n-1,2^h-1))=\nu_2(s(n,2^h))$ if $2^h \le n < 2^{h+1}$, 
which we generalize to all primes and to all ``minimum zero cases."

We have organized this paper so that the new results on the $p$-adic analysis of Stirling
numbers appear in the early sections, with the preliminaries and background in the later
sections. 

\section{$p$-adic analysis of Stirling numbers of the second kind}

Throughout this paper, $p$ = arbitrary prime and $\nu_p$ = exponential $p$-adic
valuation. We say that $r$ has a zero of order $e$ if $\nu_p(r)=e>0$, or a pole of order $e$ 
$\nu_p(r)=-e<0$. If $\nu_p(r)=0$ then $r$ is a unit. If $r\ne 0$, then 
$\epsilon_p(r)=p^{-\nu_p(r)} r$ is the unit part of $r$.

The function $\sigma_p(n)$ = sum of the base $p$ digits of $n$ plays an important role
in this paper. For $p=2$, $\sigma_2(n)$ = the number of base $2$ digits in $n$, which is 
sometimes denoted by $d_2(n)$. Obviously $\sigma_p(pn)=\sigma_p(n)$.

The connection between Stirling numbers of the second kind and higher order Bernoulli
numbers is given by 
\begin{align}
S(n,k) = \binom{n}{k} B_{n-k}^{(-k)}
\end{align}

Using the standard formula (4.3) for $\nu_p\binom{n}{m}$, the estimate of Lemma 5.1
for $\nu_p(B_n^{(l)})$ now translates to

\begin{lemma} $\nu(S(n,k)) \ge \lceil (\sigma(k)-\sigma(n))/(p-1)\rceil$ if $n \ge k$.
\end{lemma}

This lemma was  proven for $p=2$ by De Wannemacker ([7, Th. 3]). 
The proof he gave is non-trivial, involving 
Stirling number identities and induction, and doesn't appear to extend to odd primes.
Lengyel has proven an estimate for odd primes [13, Theorem 5] that is less precise
and never sharp. Note that since $\nu(S(n,k)) \in \mathbb{N}$, the estimate in
this lemma is equivalent to the estimate $\nu(S(n,k)) \ge (\sigma(k)-\sigma(n))/(p-1)$.

We define the minimum zero case for $S(n,k)$ as one where the general inequality
noted at the end of the preceding paragraph is an equality, namely
\begin{align}
S(n,k) \text{~is a minimum zero case if~} \nu(S(n,k)) = 
 (\sigma(k)-\sigma(n))/(p-1).
\end{align}

The concept of minimum zero directly relates to the concept of maximum pole
for higher order Bernoulli polynomials (5.4), which we introduced in [4]. Combining 
these definitions with the congruence in Proposition 5.1 for the higher order
Bernoulli numbers, we get the followig theorem, which establishes a simple,
effective binomial coefficient criterion.

\begin{theorem} The following are equivalent:
\begin{itemize}
\item[(i)] $S(n,k)$ is a minimum zero case,
\item[(ii)] $B_{n-k}^{(-k)}(x)$ has maximum pole.
\item[(iii)] $r=(n-k)/(p-1) \in \mathbb{N}$ and $p \nmid \binom{-(n+1)}{r}$, i.e.,
$p \nmid \binom{n+r}{r}$. 
\end{itemize}
Furthermore, in the minimum zero case, we have
\[
\epsilon(S(n,k)) \equiv (-1)^r  \epsilon(n!/k!) \binom{n+r}{r} \mod p.
\] 
\end{theorem}
Remarks. Since the classical theorems 
are all $p=2$ theorems, it is worth noting what this theorem says for $p=2$. In this
case, (iii) simply says $\binom{n+r}{r}=\binom{n+n-k}{n}$ is odd, i.e., that $n$ and
$n-k$ have no common base $2$ digits.

\begin{corollary} 
$S(n,k)$ is a minimum zero case if and only if $S(np,kp)$ is a minimum zero case.
Furthermore, if $S(n,k)$ is a minimum zero case, then $\nu(S(n,k) =\nu(S(np,kp))$
and $\epsilon(S(n,k)) \equiv \epsilon(S(np,kp)) \mod p$. 
\end{corollary}

\begin{corollary}With the same notations as in the theorem, 
if $\sigma(k)=\sigma(n)$ then 
\[
S(n,k) \equiv (-1)^r \epsilon(n!/k!)\binom{n+r}{r} \mod p.
\]
\end{corollary}

Remark. This corollary implies that if $\sigma(k)=\sigma(n)$ and $r=(n-k)/(p-1)$, 
then $p | S(n,k)$ if and only if $p|\binom{n+r}{r}$, i.e. if and only if $S(n,k)$ is not
a minimum  zero case. 

We can now easily prove an analog of De Wannemacker's Theorem valid for all primes $p$. 
The following theorem has De Wannemacker's result as the special case for $p=2$.
Even for $p=2$, the proof is much simpler than any proofs in the literature which
we know. 

\begin{theorem}Let $n=ap^h$ with $1 \le a \le p-1$ and assume that $1 \le k \le n$
and $p-1 | n-k$. Then $S(n,k)$ is a minimum zero case and
\[
\nu(S(n,k)) = \frac{\sigma(k)-\sigma(n)}{p-1} = \frac{\sigma(k)-a}{p-1}.
\]
\end{theorem}
\begin{proof} If $r =(n-k)/(p-1)$ then $r < p^h$, so $p \nmid \binom{n+r}{r}$
by the Lucas Theorem, and so we have the minimum zero
case by the preceding theorem, giving the equations of Theorem 2.2.
\end{proof}

\begin{corollary}With the same assumptions, we have 
\begin{align}
\epsilon(S(n,k)) \equiv (-1)^{r+ah}a!/ \epsilon(k!) \mod p. \notag
\end{align}
\end{corollary}
\begin{proof} We have the minimum zero case by Theorem 2.2, and $\binom{n+r}{r} \equiv
1$ mod $p$ since $r$ and $n$ have disjoint base $p$ representations. Finally,
the standard Lemma 4.1 congruence $\epsilon((ap^h)!) \equiv (-1)^{ah} a! \mod p$
and the congruence in Theorem 2.1 give the desired result. 
\end{proof}
The next theorem shows that the minimum zero Stirling numbers of the second kind
have certain invariance properties.

\begin{theorem} Let $S(n,k)$ be a minimum zero case and $0 \le b < 
\min\{p^{\nu(k)},p^{\nu(n)}\}$. 
Let $n'=n+b$ and $k'=k+b$. Then $S(n',k')$ is a minimum zero case and
\begin{itemize}
\item[(i)] $\nu(S(n',k')) =\nu(S(n,k))$.
\item[(ii)] $\epsilon(S(n',k')) \equiv \epsilon(S(n,k)) \mod p$.
\end{itemize}
\end{theorem}

\begin{proof} First observe that $b$ is a common bottom segment of the base $p$
representations of $n'$ and $k'$, and $n$ and $k$ are the respective top segments.
We have $n'-k'=n-k$, so $r=(n'-k')/(p-1) = (n-k)/(p-1)$.
Since the base $p$ representations of $n$, $k$, and $n+r$ are all disjoint from the
representation of $b$, we have $\binom{n'+r}{r} \equiv \binom{n+r}{r} 
\mod p$ by the Lucas congruence. Hence, $S(n',k')$ is also a minimum zero case.
Since  $\sigma(k')-\sigma(n')= \sigma(k)-\sigma(n)$, part (i) is now established.

For part (ii) consider
\[
(n'!/k'!) /(n!/k!) = \frac{n'!}{n! b!}\bigg/ \frac{k'!}{k! b!} = \binom{n'}{n}\bigg/ \binom{k'}{k}.
\]
But now the disjointness of $n$ and $b$ implies that $\binom{n'}{n} \equiv 1 \mod p$,
and similarly the disjointness of $k$ and $b$ implies that $\binom{k'}{k} \equiv 1
\mod p$. Hence $\epsilon(n'!/k'!) \equiv \epsilon(n!/k!) \mod p$, so by the congruence in 
Theorem 2.1, we have $\epsilon(S(n',k')) \equiv \epsilon(S(n,k)) \mod p$. 
\end{proof}

The following corollary, which stregthens DeWannemacker's Theorem, is a special case
of Theorem 2.3.

\begin{corollary} Let $n=ap^h$ with $1 \le a \le p-1$, and assume that $1 \le k\le n$
and $p-1 | n-k$. Let $n'=n+b$ and $k'=k+b$, where $0 \le b < p^{\nu(k)}$. Then
$S(n',k')$ is a minimum zero case and
\begin{itemize}
\item[(i)] $\nu(S(n,k)) = \nu(S(n',k'))=(\sigma(k)-a)/(p-1)$.
\item[(ii)] $\epsilon(S(n,k))\equiv \epsilon(S(n',k')) \mod p$.
\end{itemize}
\end{corollary}

Next we consider the central Stirling numbers $S(2k,k)$, which are close relatives of 
the Catalan numbers, and are significant for combinatorics. In [6, Th. 2.4], O-Y Chan and
D. Manna showed in a non-trivial way that $S(2k,k)$ is odd if and only if $k$ is
Fibbinary, i.e., if the base $2$ representation of $k$ has no consecutive ones. We give
a short proof of this theorem, generalized to all primes $p$.
The proof given by Chan and Manna for $p=2$, considers many parity cases. 

To generalize to arbitrary primes $p$, define $S(pk,k)$ as a $p$-central Stirling number 
and $k$ as $p$-Fibbinary if the sum of any two consecutive digits of the base $p$
representation of $k$ is at most $p-1$. These concepts clearly specialize to central
Stirling number and Fibbinary number for $p=2$. 

\begin{theorem} $p \nmid S(pk,k)$ if and only if $k$ is $p$-Fibbinary. \end{theorem}

\begin{proof} Since if $n=pk$, then $r=(n-k)/(p-1)=k$ and $\sigma_p(n)=\sigma_p(k)$.
Hence $S(pk,k)$ is a minimum zero case iff $\nu(S(pk,k))=0$, so $p \nmid S(pk,k)$
iff $p\nmid \binom{pk+k}{k}$. But by Lucas' Theorem this is equivalent to the
$p$-Fibbinary condition for $k$. 
\end{proof}

\begin{corollary}
If $k = \sum_i a_ip^i$ is the base $p$ representation, then 
\[
S(pk,k) \equiv \prod_i \binom{a_i+a_{i+1}}{a_i} \mod p.
\]
\end{corollary}
\begin{proof}
This follows immediately from the Lucas congruence for $\binom{pk+k}{k}$, with
$\epsilon((pk)!/k!) \equiv (-1)^k \mod p$.
\end{proof}

We now turn to a result conjectured by T. Amdeberhan et al in [5] and proven by
Hong et al ([8, Th. 3.2]) several years later, namely
\begin{align}
\nu_2(S(2^h+1,k+1))= \sigma_2(k)-1.
\end{align}

We give a proof of this result, which is more general since it works for all
primes $p$, and replaces the assumption that $n=2^h$ by the weaker assumption that
$S(n,k)$ is a minimum zero case. The proof is also shorter and we believe more 
instructive than the one given for the special case $p=2$ in [8].

\begin{theorem} Suppose $S(n,k)$ is a minimum zero case. Then \\
$\nu(S(n+1,k+1))=\nu(S(n,k))$ and $\epsilon(S(n+1,k+1))
\equiv \epsilon(S(n,k)) \text{~mod~} p$.
\end{theorem}
\begin{proof} By the standard recursion for Stirling numbers of the second kind, we have
\[S(n+1,k+1)=S(n,k)+(k+1)S(n,k+1)\]
Hence it will suffice for the first assertion to show that $\nu((k+1)S(n,k+1))>\nu(S(n,k))$,
by a standard property of valuations. By Lemma 2.1,\\
$\nu(S(n,k+1)) \ge \lceil (\sigma(k+1)-\sigma(n))/(p-1)\rceil$, so by Lemma 4.2 we have
\begin{align}
\nu(k+1)+\nu(S(n,k+1)) &\ge \lceil (\nu(k+1)(p-1)+\sigma(k+1)-\sigma(n))/(p-1)
 \rceil \notag\\
&= \lceil(1+\sigma(k)-\sigma(n))/(p-1)\rceil. \notag
\end{align}

But $p-1 |(\sigma(k)-\sigma(n))$ by assumption, so this number equals $1+ \nu(S(n,k))$.
The proof of the congruence now follows from (4.10). 
\end{proof}
Note that simple examples show that $S(n+1,k+1)$ may not be a minimum zero case
in Theorem 2.5. For example, for $p=2$, we have $S(5,3)$ is a minimum zero case, since $n=5$
and $n-k=2$ have no common base $2$ digit, i.e., $2\nmid \binom{n+r}{r}$. However,
$S(6,4)$ is not a minimum zero case since now $n+r=6+2$, which does have a 
base $2$ carry. 

\section{$p$-adic analysis of Stirling numbers of the first kind}

We give some results for Stirling numbers of the first kind $s(n,k)$,
which are analogous to the results for Stirling numbers of the second kind. We believe
they are all new.

We now have the connecting formula
\begin{align}
s(n,k) = \binom{n-1}{k-1}B_{n-k}^{(n)} 
\end{align}

Our first result, which is analogous to Lemma 2.1, and has essentially the same proof,
is the following.

\begin{lemma} $\nu_p(s(n,k)) \ge \lceil \frac{\sigma(k-1)-\sigma(n-1)}{p-1} \rceil$.
\end{lemma}
Remarks. In [14] Lengyel gives several striking estimates for the $p$-adic values
of $s(n,k)$, including $\nu_p(s(n,k)) \rightarrow \infty$ as $n\rightarrow\infty$
for $k$ fixed. Our methods do not suffice to yield these results. He also considers the case 
where $n-k$ is fixed, and in this case our estimate compares well with his. 

References [11, 15] extend the $p$-adic analysis of Stirling numbers of the first kind,
with [11] making heavy use of the Newton polygon of the horizontal generating
function $(x)_n$.

We can define the minimum zero case for $s(n,k)$ by
\begin{align}
\nu_p(s(n,k)) = (\sigma(k-1)-\sigma(n-1)))/(p-1).
\end{align}

Since $\nu_p(s(n,k)) \in \mathbb{N}$,
this is equivalent to sharpness of the estimate in Lemma 3.1 and to the maximum pole 
case for $B_{n-k}^{(n)}(x)$, i.e. to $\nu(B_{n-k}^{(n)}) = -\sigma(n-k)/(p-1)$.
It is also equivalent to $p \nmid \binom{k-1}{r}$,
where $r=(n-k)/(p-1) \in \mathbb{N}$.

This last formula points to an essential difference between the Stirling numbers of
the first and second kinds, namely the minimum zero case here requires that
$r \le k-1$ since $p\nmid \binom{k-1}{r}$, so $k \le n < kp$ is a necessary condition
for the Stirling number $s(n,k)$ to be a minimum zero case. There is nothing
comparable for Stirling numbers of the second kind.

We get the following theorem, essentially by definition.

\begin{theorem} If $r=(n-k)/(p-1)$, then in the minimum zero case
\[
\nu(s(n,k))=(\sigma(k-1)-\sigma(n-1))/(p-1)
\]
and
\[
\epsilon(s(n,k)) \equiv \epsilon((n-1)!/(k-1)!) \binom{k-1}{r} \mod p.
\]
\end{theorem}

\begin{corollary}
$s(n,k)$ is a minimum zero case if and only if $s(np,kp)$ is a minimum zero case. 
Furthermore, if $s(n,k)$ is a minimum zero case, then $\epsilon(s(n,k)) \equiv
\epsilon(s(np,kp)) \mod p$. 
\end{corollary}

We have a theorem for Stirling numbers of the first kind analogous to De
Wannemacker's Theorem, generalized to arbitrary primes.

\begin{theorem} Let $k$ have a single base $p$ digit, i.e. $k=ap^h$ with
$1 \le a \le p-1$. Then the minimum zero case holds for all $s(n,k)$ with
$k \le n < kp$ such that $p-1|n-k$.
\end{theorem}
\begin{proof} If $r=(n-k)/(p-1)$ then clearly $ r \le k-1$ since $n-k<k(p-1)$ which implies by 
Lucas's Theorem that $p\nmid\binom{k-1}{r}$, since $k-1 = (a-1)p^h + \linebreak
(p-1)p^{h-1}+ \cdots + (p-1)$.
\end{proof}

\begin{corollary} With the same assumptions and notations
\[
\nu(s(n,ap^h)) = \frac{a-1-\sigma(n-1)}{p-1} +h.
\]
and
\[
\epsilon(s(n,ap^h)) \equiv (-1)^{ah+r-r_h} \frac{\epsilon((n-1)!)}{(a-1)!}
 \binom{a-1}{r_h} \mod p,
\]
where $r_h$ is the coefficient of $p^h$ in the base $p$ representation of $r$.
\end{corollary}
\begin{proof} $\sigma(k-1)$ is given in the above proof, namely $\sigma(k-1) = a-1+h(p-1)$,
which gives the first part. 
For the congruence part, use Lemma 4.1 applied to $\epsilon(k!)$ with $(k-1)! = k!/k$,
together with the Lucas congruence with the last line of the preceding proof, 
and the fact that $\binom{p-1}{r_i} \equiv 
\binom{-1}{r_i} = (-1)^{r_i}$ mod $p$, for each digit $r_i$ of $r$, together with
$\sigma(r) \equiv r \mod p-1$, so $\sigma(r)$ and $r$ have the same parity if $p \ne 2$.
\end{proof}

Remark. The presence of $h$ in $\nu(s(n,ap^h))$ is different from the situation for
$\nu(S(ap^h,k))$, and illustrates that the Stirling numbers of the first and second
kind have different character. 

The special case for $p=2$ is particularly simple and is worth noting.

\begin{corollary} Let $k=2^h$. Then if $2^h \le n < 2^{h+1}$, we have
\begin{align}
\nu_2(s(n,k)) = h-\sigma_2(n-1). \notag
\end{align}
\end{corollary}

We have an invariance property for Stirling numbers of the first kind analogous
to the Stirling numbers of the second kind. The proof is essentially similar, and 
we will omit it. 

\begin{theorem}
Let $s(n,k)$ be a minimum zero case. Assume that $p^{\nu(t)}>n$. Let
$n'=t+n$ and $k'=t+k$. Then $s(n',k')$ is a minimum zero case and
\begin{itemize}
\item[(i)] $\nu(s(n',k'))=\nu(s(n,k))$.
\item[(ii)] $\epsilon(s(n',k')) \equiv \epsilon(s(n,k)) \mod p$. 
\end{itemize}
\end{theorem}

In this case $t$ is the common top segment of $n'$ and $k'$, and $n$ and $k$
are the respective bottom segments. 

The special case when $k=ap^h$ with $1 \le a \le p-1$ and $pk>n \ge k$ and
$p-1|n-k$, has the same invariance, which is a strengthening of the analog of
DeWannemacker's Theorem for Stirling numbers of the first kind. 

Finally we prove an analog of the Hong, Zhao and Zhao result for Stirling numbers
of the first kind, also valid for all primes $p$, and generalized to minimum zero
cases.

\begin{theorem} Let $s(n,k)$ be a minimum zero case. Then
\begin{align}
\nu(s(n-1,k-1)) &= \nu(s(n,k)) \quad \text{and} \notag\\
\epsilon(s(n-1,k-1)) &\equiv \epsilon(s(n,k)) \text{~mod~} p. \notag
\end{align}
\end{theorem}
\begin{proof} This is entirely analogous to the previous proof for the Stirling numbers of the
second kind, now using the basic recursion
\begin{align}
s(n,k)=s(n-1,k-1)-(n-1)s(n-1,k). \notag
\end{align}
\end{proof}

The rest of the proof is essentially the same as in Theorem 2.5, so we omit the details.

Observe that $s(n-1,k-1)$ may not be a minimum zero case. 

\section{$p$-adic preliminaries}

We now collect, for reference purposes, some useful standard and elementary
$p$-adic results. 

\begin{align}
n \equiv \sigma_p(n) \mod (p-1), \text{~i.e.~} p-1 | (n-\sigma_p(n)).
\end{align}

This paper makes heavy use of standard results on factorials and binomial coefficients,
which we now summarize:
\begin{align}
\nu_p(n!) &= (n-\sigma_p(n))/(p-1). \\
\nu_p\binom{n}{m} &= (\sigma_p(m)+\sigma_p(n-m)-\sigma_p(n))/(p-1).
\end{align}

Remark. From (4.2) and (4.3), it immediately follows that if $p-1|n-k$ then 
$\binom{n}{k}=\epsilon(n!/k!) p^{(n-k)/(p-1)} p^{(\sigma(k)-\sigma(n))/(p-1)}/(n-k)!$.

It also follows that $\nu_p\binom{n}{m}$ = number of carries for the base $p$ addition
of $m$ and $n-m$, whence we have the Lucas Theorem that
\begin{align}
p\nmid \binom{n}{m} \text{~iff~} n_i \ge m_i \text{~for all base~} p \text{~digits}.
\end{align}

In fact, the Lucas congruence says
\begin{align}
\binom{n}{m} \equiv \prod_i \binom{n_i}{m_i} \mod p.
\end{align}

An important special case of the Lucas congruence is that if $r$ and $n$ have disjoint
base $p$ representations then
\begin{align}
\binom{n+r}{r} \equiv 1 \mod p.
\end{align}

There is a more subtle congruence discovered by H. Anton in 1869 that if $\nu_p
\binom{n}{m}=e$ and $r=n-m$ then 
\begin{align}
 \frac{(-1)^e}{p^e} \binom{n}{m} \equiv \prod \frac{n_i!}{m_i! r_i!} \mod p
\end{align}
where $n_i,m_i,r_i$ are the base $p$ digits of $n, m, r$ respectively. 
This is a mod $p$ congruence for $\epsilon_p\binom{n}{m}$, up to the sign $(-1)^e$.

Since the base $p$ digits of $np$ are the same as those of $n$ shifted one place to the
left, it follows immediately from the Lucas and Anton congruences that 
\begin{align}
\nu_p\binom{np}{mp} = \nu_p\binom{n}{m} \text{~and~} \epsilon_p\binom{np}{mp}
 \equiv \epsilon_p\binom{n}{m} \mod p.
\end{align}

If $p$ is understood by the context, we may suppress the $p$ in our notations, i.e. use
$\nu, \sigma, \epsilon$ instead of $\nu_p, \sigma_p,\epsilon_p$ respectively. 

Finally, we make frequent use of the formula 
\begin{align}
\binom{-a}{r} = (-1)^r \binom{a+r-1}{r}, \text{~i.e.~} \binom{-(n+1)}{r} =(-1)^r
 \binom{n+r}{r}.
\end{align}

By basic properties of valuations, it is clear that $\epsilon(ab)=\epsilon(a)\epsilon(b)$,
and 
\begin{align}
\text{if~} \nu(a) < \nu(b), \text{~then~}
\nu(a+b)=\nu(a) \text{~and~} \epsilon(a+b)\equiv\epsilon(a) \mod p.
\end{align}

Remark. It is worth noting that $c/d \equiv 1 \mod p$ if and only if
$\nu(c)=\nu(d)$ and $\epsilon(c) \equiv \epsilon(d) \mod p$. 

These observations lead immediately to the following lemma. We omit the proof, 
which is a straightforward generalization of Wilson's Theorem and proof.

\begin{lemma} Assume $1 \le a \le p-1$. Then
\[
\epsilon((ap^h)!) \equiv (-1)^{ah}a! \mod p.
\]
\end{lemma}

It is also well-known and easy to prove that
\begin{align}
\epsilon_p((pk)!) \equiv (-1)^k\epsilon_p(k!) \text{~mod~} p.
\end{align}

Finally we conclude with a useful, elementary lemma.

\begin{lemma} $\sigma_p(k+1)=\sigma_p(k)+1-(p-1)u$ where $u=\nu_p(k+1)$ = 
number of consecutive digits at the bottom of the base $p$ representation
of $k$ which are equal to $p-1$.
\end{lemma}
\begin{proof} The effect of adding one to $k$ is to replace the bottom $u$ digits by
zeros and increase the next digit by one.
\end{proof}

\section{Background on Stirling numbers and higher order Bernoulli numbers and 
polynomials}

If $n \in \mathbb{N}$ and $l \in \mathbb{Z}$, the Bernoulli polynomials $B_n^{(l)}(x)$
of order $l$ and degree $n$ are defined by
\begin{align}
\left( \frac{t}{e^t-1}\right)^l e^{tx} = \sum_{n=0}^\infty B_n^{(l)}(x) \frac{t^n}{n!}.
\end{align}

The higher order Bernoulli numbers are the constant terms $B_n^{(l)}=B_n^{(l)}(0)$. 
The polynomial $B_n^{(l)}(x) \in \mathbb{Q}[x]$ is monic with degree $n$.

The Stirling numbers of the first kind $s(n,k)$ can be defined by
\begin{align}
(x)_n = \sum_{k=1}^\infty s(n,k)x^k
\end{align}
where $(x)_n=x(x-1) \cdots (x-(n-1)) =n! \binom{x}{n}$.

The $s(n,k)$ are integers and the sign of $s(n,k)$ is $(-1)^{n-k}$. The unsigned Stirling
numbers $|s(n,k)|$ count the number of $n$-permutations with $k$ cycles.

The Stirling numbers of the second kind $S(n,k)$ can be defined combinatorially by
\begin{align}
S(n,k) = \text{~number of partitions of an $n$-set into $k$ subsets.}
\end{align}

Remarks. We showed in [1] how to precisely locate the successively increasing order
poles of the coefficients of $B_n^{(l)}(x)$, arranged from top degree down, 
 which we call the poles of $B_n^{(l)}(x)$,
and we showed that these poles have a remarkably regular pattern. 
The salient features of the pole pattern are that 
the first pole has order $1$, the next bigger pole has order
$2$, etc., and that all these first occurrences appear in codegrees $i$, where
$p-1|i$ and $p\nmid \binom{n}{i}$.

Subsequently
in [4] we interpreted these results in terms of the Newton polygon of $B_n^{(l)}(x)$
and gave a precise, algorithmic, description of the descending portion of this Newton
polygon, which summarizes the pole pattern.

The following lemma was proven in [1], and by a different method, also in~[3].

\begin{lemma} 
\begin{align}\nu(B_n^{(l)}) \ge -\lfloor \sigma(n)/(p-1)\rfloor. \notag
\end{align}.
\end{lemma}

We were also able to prove some general congruences for the higher order
Bernoulli numbers $B_n^{(l)}$ in [3]. We will generally assume that $p-1 |n$ (or \linebreak
$p-1|n-k$ for the applications to Stirling numbers $S(n,k)$ and $s(n,k)$),
since that is simplest.
The following proposition is the special case of [3, Th. 1] where $p-1|n$, with
some notational changes.

\begin{proposition} Suppose $p-1|n$ and let $r=n/(p-1)$. Then
\[
(-1)^n p^r B_n^{(l)}/n! \equiv (-1)^r\binom{n+r-l}{r} \mod p.
\]
\end{proposition}

Note that since $p-1|n$, we can omit the factor $(-1)^n$ from the preceding 
congruence.

We introduced the concept of maximum pole in [4] for $B_n^{(l)}(x)$ by
\begin{align}
\nu_p(B_n^{(l)}) = -\sigma(n)/(p-1),
\end{align}
which is the theoretical minimum value and obviously is only attainable if $p-1|n$. 
Observe that if $p-1 |n$, this is equivalent to sharpness of the estimate in Lemma 5.1.
In the maximum pole case, $B_n^{(l)}(x)$ has a pole if $n>0$, which is the biggest
pole for all the coefficients of $B_n^{(l)}(x)$. This occurs when the Newton polynomial
of $B_n^{(l)}(x)$ is strictly decreasing.
By the preceding analysis, there is a maximum pole iff
\begin{align}
r= n/(p-1) \in \mathbb{N} \quad \text{and} \quad p\nmid \binom{l-n-1}{r} .
\end{align}
In the maximum pole case, we have the nontrivial congruences 
\begin{align}
p^rB_n^{(l)}/n! \equiv (-1)^n\binom{l-n-1}{r} \equiv (-1)^r\binom{n-l+r}{r}
 \mod p .
\end{align}

\pagebreak

\centerline{ACKNOWLEDGEMENTS}
~
The author thanks E. Herman for his help in preparing this manuscript. The
author also thanks T. Lengyel for his generous advice and encouragement. \\
\\
\centerline{REFERENCES}
~
1. A. Adelberg, On the degrees of irreducible factors of higher order Bernoulli 
polynomials, \textit{Acta Arith.} \textbf{62} (1992), 329-342.\\
2. A. Adelberg, A finite difference approach to degenerate Bernoulli and
Stirling polynomials, \textit{Discrete Math.} \textbf{140} (1995), 1-21.\\
3. A. Adelberg, Congruences of $p$-adic integer order Bernoulli numbers,
\textit{J. Number Theory} \textbf{59} No. 2 (1996), 374-388.\\
4. A. Adelberg, Higher order Bernoulli polynomials and Newton polygons,
G. E. Bergum et al (eds.), \textit{Applications of Fibonacci Numbers} \textbf{7}
(1998), 1-8.\\
5. T. Amdeberhan, D. Manna and V. Moll, The $2$-adic valuation of Stirling
numbers, \textit{Experimental Math.} \textbf{17} (2008), 69-82.\\
6. O-Y. Chan and D. Manna, Divisibility properties of Stirling numbers of the
second kind, Proceedings of the Conference on Experimental Math., T. 
Amdeberhan, L. A. Medina, and V. Moll eds., \textit{Experimental Math.}
(2009).\\
7. S. De Wannemacker, On $2$-adic orders of Stirling numbers of the second
kind, \textit{Integers Electronic Journal of Combinatorial Number Theory},
\textbf{5} (1) (2005), A21, 7 pp. (electronic).\\
8. S. Hong J. Zhao, and W. Zhao, The $2$-adic valuations of Stirling numbers of
the second kind, \textit{Int. J. Number Theory} \textbf{8} (2012), 1057-1066.\\
9. S. Hong, J. Zhao, and W. Zhao, Divisibility by $2$ of Stirling numbers of the
second kind and their differences, \textit{J. Number Theory} \textbf{140}
(2014), 324-348.\\
10. S. Hong, J. Zhao, and W. Zhao, The $2$-adic valuations of differences of
Stirling numbers of the second kind, \textit{J. Number Theory} \textbf{153}
(2015), 309-320.\\
11. T. Komatsu and P. T. Young, Exact $p$-adic valuations of Stirling numbers
of the first kind, \textit{J. Number Theory} \textbf{177} (2017), 20-27.\\ 
12. T. Lengyel, On the divisibility by $2$ of the Stirling numbers of the second
kind, \textit{Fibonacci Quart.} \textbf{32} (3) (1994), 194-201.\\
13. T. Lengyel, Alternative proofs on the $2$-adic order of Stirling numbers
of the second kind, \textit{Integers} \textbf{10} (2010), A38, 453-468.\\
14. T. Lengyel, On $p$-adic properties of the Stirling numbers of the first kind,
\textit{J. Number Theory} \textbf{148} (2015), 73-94.\\
15. P. Leonetti and C. Sanna, On the $p$-adic valuation of Stirling numbers of
the first kind, \textit{Acta Math. Hungar.} \textbf{151} (2017), 217-231.

\end{document}